\documentclass[12pt, leqno]{amsart}
\usepackage{amsmath, amsthm, amscd, amsfonts, amssymb, graphicx, color}
\usepackage[bookmarksnumbered, colorlinks, plainpages]{hyperref}

\newtheorem{theorem}{Theorem}[section]
\newtheorem{lemma}[theorem]{Lemma}
\newtheorem{proposition}[theorem]{Proposition}
\newtheorem{corollary}[theorem]{Corollary}
\theoremstyle{definition}
\newtheorem{definition}[theorem]{Definition}

\theoremstyle{remark}
\newtheorem{remark}[theorem]{Remark}
\numberwithin{equation}{section}

\newcommand{\thmref}[1]{Theorem~\ref{#1}}

\newcommand{\lemref}[1]{Lemma~\ref{#1}}
\newcommand{\propref}[1]{Proposition~\ref{#1}}

\newcommand{\defnref}[1]{Definition~\ref{#1}}

\newcommand{\remref}[1]{Remark~\ref{#1}}

\def\Z{{\bf{Z}}}
\def\N{{\bf{N}}}

\def\ra{\rightarrow}

\def\mk{\medskip}

\def\tfae{the following conditions are equivalent:}

\begin{document}
\setcounter{page}{1}

\title[Topological amenability]{Topological amenability is a Borel property}

\author[Jean Renault]{Jean Renault}

\address{Universit\'e d'Orl\'eans and CNRS  (MAPMO/UMR 7349 and FDP/FR2964), D\'epartment de Math\'ematiques, 
45067 Orl\'eans Cedex 2, France.}
\email{\textcolor[rgb]{0.00,0.00,0.84}{Jean.Renault@univ-orleans.fr}}






\begin{abstract}
We establish that a $\sigma$-compact locally compact groupoid possessing a continuous Haar system is topologically amenable if and only if it is Borel amenable. We give some examples and applications.
\end{abstract}

\maketitle

\section{Introduction}
The notion of topological amenability of a locally compact groupoid $G$ endowed with a Haar system was first introduced in \cite[Definition II.3.6]{ren:approach} as a convenient sufficient condition for measurewise amenability. Indeed, it implies both the equality of the reduced C*-algebra $C_r^*(G)$ and the full C*-algebra $C^*(G)$ of the groupoid and the nuclearity of $C^*(G)$. However, some later results have given a greater interest to this notion. When $G$ is an \'etale Hausdorff locally compact groupoid, one has a direct equivalence between the topological amenability of $G$ and the nuclearity of $C^*_r(G)$ (see \cite{ad:syst} in the case of discrete group actions and \cite[Theorem 5.6.18]{bo:fda} in the general \'etale case). Moreover, topological amenability has applications to the Baum-Connes conjecture: for example, J.-L. Tu shows in \cite{tu:amenable} that topologically amenable Hausdorff locally compact groupoids with Haar systems admit proper affine actions on Hilbert bundles, hence satisfy the Baum-Connes conjecture. The equivalence between topological amenability and measurewise amenability is established in \cite[Corollary 3.3.8]{adr:amenable} for a large class of groupoids, including \'etale groupoids. The proof given in \cite{adr:amenable} relies on unnecessary assumptions which obscure its main ideas. In particular, it misses a notion of Borel amenability analogous to the above notion of topological amenability. The adequate notion of Borel amenability appears explicitly shortly later in Section 2.4 of the comprehensive work \cite{jkl:relations} by S. Jackson, A. S. Kechris and A. Louveau about countable Borel equivalence relation. It turns out that for $\sigma$-compact locally compact groupoids with Haar systems, both notions coincide. Although this may be well-known to specialists, it seems useful to present here a general proof of this fact. On one hand, it gives a further justification to the early definition of topological amenability. On the other, it has practical applications since Borel amenability is easier to check than topological amenability; this will be illustrated by some examples in the second section. As in the case of groups, where it essentially amounts to the equivalence of amenability and Reiter's properties $(P_1)$ and $(P_1^*)$, the crux of the proof is a classical application of the Hahn-Banach theorem to the closure of a convex set. Our proof is modelled after the group case (see \cite[Theorem G.3.1]{bhv:T} for a recent exposition). The definition of topological amenability given below could be adapted to arbitrary topological groupoids. However, the proof of the equivalence makes an essential use of the existence of a continuous Haar system and of the locally compact topology of $G$. Moreover it is not clear how useful this notion and its Borel counterpart are for non locally compact groups.  The next section contains the definition of topological and Borel amenability and the main result, namely the equivalence of these notions for $\sigma$-compact locally compact groupoids endowed with a Haar system. The last section contains applications and examples which take advantage of the flexibility provided by the equivalence of both notions. 

We use the terminology and the notation of \cite{adr:amenable}. The unit space of a groupoid $G$ is denoted by $G^{(0)}$. The elements of $G$ are usually denoted by $\gamma, \gamma',\ldots$; those of $G^{(0)}$ are denoted by $x,y,\ldots$. The structure of $G$ is defined by the inclusion map $i: G^{(0)}\ra G$ (we shall identify $x$ and $i(x)$), the range and source maps $r,s:G\ra G^{(0)}$, the inverse map $\gamma\mapsto\gamma^{-1}$ from $G$ to $G$ and the multiplication map $(\gamma,\gamma')\mapsto \gamma\gamma'$ from the set of composable pairs 
$$G^{(2)}=\{(\gamma,\gamma')\in G\times G: s(\gamma)=r(\gamma')\}$$
to $G$. Given $A,B\subset G^{(0)}$, we write $G^A=r^{-1}(A)$, $G_B=s^{-1}(B)$ and $G^A_B=G^A\cap G_B$. Similarly, given $x,y\in G^{(0)}$, we write $G^x=r^{-1}(x)$, $G_y=s^{-1}(y)$ and $G(x)=G^x_x$.  A Borel [resp. topological] groupoid is a groupoid endowed with a compatible Borel [resp. topological] structure: $G$ and $G^{(0)}$ are Borel [resp. topological] spaces and the above maps are Borel [resp. continuous]. We need to be more precise in the definition of a topological groupoid: we assume that $G^{(0)}\subset G$ and $G^{(2)}\subset G\times G$ have the subspace topology. We also include in the definition of a topological groupoid the assumptions that the unit space is Hausdorff and that the range and source maps are open but we do not assume that $G$ is Hausdorff. Foliation theory, where the notion of amenability is preeminent, provides many examples of non-Hausdorff locally compact groupoids which should be covered by our discussion. With respect to amenability, non-Hausdorff groupoids do not present real difficulties but make the exposition more technical. It may help on a first reading to assume that groupoids are Hausdorff. The articles \cite{pat:gpd, ks:regular,tu:non-hausdorff,mw:equivalence} contain some of the technical tools needed in the non-Hausdorff case. As in \cite{mw:equivalence}, we do not include Hausdorffness in the definition of a compact space (our compact spaces are called quasi-compact in Bourbaki's terminology). By definition, a not necessarily Hausdorff locally compact space is a topological space such that every point admits a compact Hausdorff neighborhood. Equivalently, it is a topological space which admits a cover by locally compact Hausdorff open subsets. This second definition provides a convenient bridge from Hausdorff locally compact spaces to non-Hausdorff locally compact spaces. Given a locally compact Hausdorff open subset $U$ of a locally compact space $X$, $C_c(U)$ denotes the usual space of complex-valued continuous functions on $U$ which have compact support. When one extends by 0 outside $U$ a function $f\in C_c(U)$, the resulting extension $\tilde f$ is not necessarily continuous on $X$. Following A. Connes, ${\mathcal C}_c(X)$ denotes the linear span of these functions. We keep the usual definition of a Radon measure on $X$ as a linear functional on $\mathcal{C}_c(X)$ which is continuous for the inductive limit topology. As in the Hausdorff case, a Radon measure $\nu$ defines a complex so-called Borel Radon measure, still denoted by $\nu$, on the Borel subsets contained in compact subsets (see \cite{mw:equivalence}); moreover, a linear functional on ${\mathcal C}_c(X)$ which is positive on positive functions is a Radon measure. The definition of a $\sigma$-compact locally compact space $X$ is the usual one, namely there exists an increasing sequence $(K_n)$ of compact subsets such that $X=\bigcup K_n$. As in the Hausdorff case, second countable locally compact spaces are $\sigma$-compact. We shall also use some results  from \cite[Chapters 1 and 2]{adr:amenable} which were given for Hausdorff spaces and Hausdorff groupoids and which we will adapt to the non-Hausdorff case. 

\section{Borel versus topological amenability}

Let us first give our definitions of amenability for groupoids. The definition of Borel amenability  given below is exactly the definition of 1-amenability of \cite[Definition 2.12]{jkl:relations} in the case of countable Borel equivalence relations.

\begin{definition}\label{Borel amenable} A Borel groupoid $G$ is said to be {\it Borel amenable} if there exists a {\it Borel approximate invariant mean}, i.e. a sequence $(m_n)_{n\in\N}$, where each $m_n$ is a family $(m_n^x)_{x\in G^{(0)}}$ of finite positive measure $m_n^x$ of mass not greater than one on $G^x=r^{-1}(x)$ such that:
\begin{enumerate}
\item for all $n\in\N$, $m_n$ is Borel in the sense that for all bounded Borel functions $f$ on $G$, $x\mapsto\int f dm_n^x$ is Borel;
\item  $\|m_n^x\|_1\to 1$ for all $x\in G^{(0)}$;
\item $\|\gamma m_n^{s(\gamma)}-m_n^{r(\gamma)}\|_1\to 0$ for all $\gamma\in G$.
\end{enumerate}
\end{definition}

In the above definition as well as in the rest of the paper, $\|\nu\|_1$ designates the total variation (i.e. the mass of  its absolute value $|\nu|$) of a complex bounded measure $\nu$. If there exists a Borel family $m=(m^x)$ of probability measures $m^x$ on $G^x$, one can replace condition $({\rm ii})$ by condition:

(ii')  for all $n\in\N$ and all $x$, $m_n^x$ is a probability measure.
It suffices to replace $m_n^x$ by $m_n^x/\|m_n^x\|_1$ if $\|m_n^x\|_1$ is non zero and by $m^x$ otherwise.

\begin{remark} This definition makes sense for arbitrary Borel groupoids and, in particular, for non locally compact groups. However, in the case of a non locally compact topological group $G$, it is strictly stronger than the classical definition, which is the existence of a left invariant mean on the Banach space ${\rm UCB}(G)$ of all left uniformly continuous bounded functions on $G$. I owe the following example to V. Pestov (see \cite[Remark G.3.7]{bhv:T} for references). The unitary group $U({\mathcal H})$ of an infinite-dimensional Hilbert space ${\mathcal H}$, endowed with the weak operator topology, is amenable in the classical sense. However it is not Borel amenable in the above sense. Indeed Borel amenability is inherited by virtual subgroups while $U({\mathcal H})$ contains the free group ${\bf F}_2$ as a discrete subgroup.
\end{remark}

A {\it Borel Haar system} $\lambda$ for a Borel groupoid $G$ is a family $(\lambda^x)_{x\in G^{(0)}}$ of non-zero measures on the fibers $G^x$ such that
\begin{itemize}
\item it is Borel in the sense that for all non-negative Borel functions $f$ on $G$, $x\mapsto \int f d\lambda^x$ is Borel;
\item it is left invariant in the sense that for all $\gamma\in G,\quad\gamma\lambda^{s(\gamma)}=\lambda^{r(\gamma)}$;
\item it is proper in the sense that $G$ is the union of an increasing sequence $(A_n)_{n\in\N}$ of Borel subsets such that for all $n\in\N$, the functions $x\mapsto \lambda^x(A_n)$ are bounded on $G^{(0)}$.
\end{itemize}
As it is well-known, locally compact groups have a Borel Haar system (in that case, a single measure) and the converse is essentially true. Therefore, the groupoids of a Borel action of a locally compact group on a Borel space have a Borel Haar system. Another important class of Borel groupoids with Borel Haar systems are the countable standard Borel groupoids, i.e. such that the Borel structure is standard and the range map is countable to-one. Then the counting measures $\lambda^x$ on the fibers  $G^x$ form a Borel Haar system. The countable standard Borel groupoids include the countable discrete groups and the countable standard Borel equivalence relations. In presence of a Haar system, it is known that  the approximate invariant means of the above definition can be chosen with a density with respect to the Haar system. We recall this fact below.

\begin{definition}\label{Borel approximate density} Let $G$ be a Borel groupoid equipped with a Borel Haar system $\lambda$.
A {\it Borel approximate invariant density} is a sequence $(g_n)_{n\in\N}$ of non-negative Borel functions on $G$ such that
\begin{enumerate}
\item $\int g_n d\lambda^x\le 1 ,\quad \forall x\in G^{(0)},\quad\forall n\in\N$;
\item $\int g_n d\lambda^x\to 1$ for all $x\in G^{(0)}$;
\item $\int| g_n(\gamma^{-1}\gamma_1)-g_n(\gamma_1)|d\lambda^{r(\gamma)}(\gamma_1)\to 0$ for all $\gamma\in G$.
\end{enumerate}
\end{definition}

Thus one has the following proposition (essentially \cite[Proposition 2.2.6]{adr:amenable}).

\begin{proposition}\label{Borel density and mean} A Borel groupoid $G$ equipped with a Borel Haar system $\lambda$ is Borel amenable if and only if it has a Borel approximate invariant density.
\end{proposition}

\begin{proof}
Given a Borel approximate invariant density $(g_n)$, one defines the measures $m_n^x=g_n\lambda^x$. Since
$$\|m_n^x\|_1=\int g_nd\lambda^x\quad{\rm and}\,\quad \|\gamma m_n^{s(\gamma)}-m_n^{r(\gamma)}\|_1=\int| g_n(\gamma^{-1}\gamma_1)-g_n(\gamma_1)|d\lambda^{r(\gamma)}(\gamma_1)$$
$(m_n)$ is a Borel approximate invariant mean. Conversely, let $(m_n)$ be a Borel approximate invariant mean. According to \cite[Lemma I.3]{ac:integration}, there exists a non-negative Borel function $f$ such that $\int f d\lambda^x=1$ for all $x\in G^{(0)}$. Define the non-negative Borel function $g_n$ on $G$ by
$$g_n(\gamma)=\int f({\gamma'}^{-1}\gamma)dm_n^{r(\gamma)}(\gamma').$$
Using Fubini's theorem and changes of variable, one obtains
$$\int g_n d\lambda^x=\|m_n^x\|_1$$
and
$$\int |g_n(\gamma^{-1}\gamma_1)-g_n(\gamma_1)|d\lambda^{r(\gamma)}(\gamma_1)\le\|\gamma m_n^{s(\gamma)}-m_i^{r(\gamma)}\|_1.$$
This shows that $(g_n)$ is a Borel approximate invariant density.
\end{proof}

With an abuse of language, we shall also call $(g_n)$ a Borel approximate invariant mean.

\begin{remark} Various definitions of amenability for countable Borel equivalence relations are given by Jackson, Kechris and Louveau in \cite{jkl:relations} as well as relations between them. Our definition of Borel amenability is {\it 1-amenability} of \cite[Definition 2.12]{jkl:relations}. Replacing the sequence by a net, these authors define a hierarchy of amenability properties according to the nature of the net and the more general notion of {\it Fr\'echet-amenability}. A countable Borel equivalence relation $E$ on a standard Borel space $X$ is called {\it hyperfinite} if it is an increasing union of a sequence of Borel sub-equivalence relations $E_n$ which are finite (meaning that the range map is finite-to-one). Following \cite[Definition 2.7]{jkl:relations}, it is called  {\it measure-amenable}  if there exists a universally measurable invariant mean, i.e. a family $(m^x)_{x\in X}$, where for all $x\in X$, $m^x$ is a mean on $L^\infty(E^x,\lambda^x)=\ell^\infty([x])$, $m^x=m^y$ if $(x,y)\in E$, such that for every standard Borel space $Z$ and every bounded Borel function $f$ on $X\times Z$, the map $(x,z)\mapsto \int f(y,z)dm^x(y)$ is universally measurable on $X\times Z$. Finally, a countable Borel equivalence relation $(E,X)$ is called {\it measurewise amenable} if for all measures $\mu$, the measured equivalence relation $(E,X,\mu)$ is amenable in the sense of Zimmer. Here are some of the implications for countable standard Borel equivalence relations (see \cite[Proposition 2.13]{jkl:relations}): hyperfiniteness $\Rightarrow$ 1-amenability $\Rightarrow$ Fr\'echet-amenability. Under the continuum hypothesis, Fr\'echet-amenability $\Rightarrow$ measure-amenability. It is also known \cite{kec:amenable} that under the continuum hypothesis, measure-amenability is equivalent to measurewise amenability.

\end{remark}

 Let us turn now to the topological setting.

\begin{definition}\label{topologically amenable} A locally compact groupoid $G$ is said to be {\it topologically amenable} if there exists a {\it topological approximate invariant mean}, i.e. a sequence $(m_n)_{n\in\N}$, where each $m_n$ is a family $(m_n^x)_{x\in G^{(0)}}$, $m_n^x$ being a finite positive measure of mass not greater than one on $G^x=r^{-1}(x)$ such that
\begin{enumerate}
\item for all $n\in\N$, $m_n$ is continuous in the sense that for all $f\in \mathcal{C}_c(G)$, $x\mapsto\int f dm_n^x$ is continuous;
\item $\|m_n^x\|_1\to 1$ uniformly on the compact subsets of $G^{(0)}$;
\item $\|\gamma m_n^{s(\gamma)}-m_n^{r(\gamma)}\|_1\to 0$ uniformly on the compact subsets of $G$.
\end{enumerate}
\end{definition}

Let us compare this definition and \cite[Definition 2.2.2]{adr:amenable}. There, one has a net $(m_i)_{i\in I}$ rather than a sequence $(m_n)_{n\in\N}$ and the measures $m_i^x$ are required to be probability measures. If $G$ is $\sigma$-compact, the net can be replaced by a sequence. In the other direction, as in the Borel case, one can normalize the families $m_n$ of the above definition to obtain continuous families of probability measures $m'_n$ satisfying the approximate invariance property (iii). Thus both definitions give the same notion of topological amenability when $G$ is $\sigma$-compact.

\vskip 5mm
We recall that a (continuous) {\it Haar system} is a family $(\lambda^x)_{(x\in G^{(0)}}$ of Radon measures on the fibers $G^x$ (which are locally compact and Hausdorff according to \cite{tu:amenable}) satisfying the above continuity assumption and the left invariance property $\gamma\lambda^{s(\gamma)}=\lambda^{r(\gamma)}$ for all $\gamma\in G$. We have seen that in presence of a Haar system, we can assume that the approximate invariant means have a density with respect to the Haar system. Our stronger assumptions lead to the following definition:

\begin{definition}\label{topological approximate density} Let $G$ be a locally compact groupoid equipped with a continuous Haar system $\lambda$.
A {\it topological approximate invariant density} is a sequence $(g_n)$ in $\mathcal{C}_c(G)^+$ such that
\begin{enumerate}
\item $\int g_n(x)d\lambda^x\le 1 ,\quad \forall x\in G^{(0)},\quad\forall i$;
\item $\int g_n(x)d\lambda^x\to 1$ uniformly on every compact subset of $G^{(0)}$;
\item  $\int| g_n(\gamma^{-1}\gamma_1)-g_n(\gamma_1)|d\lambda^{r(\gamma)}(\gamma_1)$ tends to 0  uniformly on every compact subset of $G$.
\end{enumerate}
\end{definition}

The same proof as in the Borel case gives:

\begin{proposition}\cite[Proposition 2.2.13]{adr:amenable}\label{topological mean and density} A locally compact groupoid $G$ equipped with a continuous Haar system $\lambda$ is topologically amenable if and only if it has a topological approximate invariant density.
\end{proposition}

Again, we shall also call $(g_n)$ as above a topological approximate invariant mean. Proposition 2.2.13 of \cite{adr:amenable} gives the equivalence of the notion of topological amenability used in the present article and the original definition which appears after Definition 2.3.6, page 92, of  \cite{ren:approach} (in the case when $G$ is $\sigma$-compact since we consider sequences only).

\vskip 5mm
Let $G$ be a locally compact groupoid endowed with a continuous Haar system $\lambda$. We define the Banach space ${\mathcal E}$ as the completion of the linear space ${\mathcal C}_c(G)$ with respect to the norm
$$\|f\|=\sup_{x\in G^{(0)}}\int |f(\gamma)| d\lambda^x(\gamma).$$
It is useful to view ${\mathcal E}$ as the space of continuous sections vanishing at infinity of a Banach bundle over $G^{(0)}$. We denote by $L^1(G,\lambda)$ the Banach bundle which has $L^1(G^x,\lambda^x)$ as fiber above $x\in G^{(0)}$ and ${\mathcal C}_c(G)$ as total space of continuous sections. Given $f\in  {\mathcal C}_c(G)$ and $x\in G^{(0)}$, we denote by $f_{|x}$ its restriction to $G^x$. The Banach bundle $L^1(G,\lambda)$ is upper semi-continuous in the sense that the functions $x\mapsto \|f_{|x}\|_x=\int |f| d\lambda^x$ are upper semi-continuous for all $f\in  {\mathcal C}_c(G)$  (see \cite[Lemma 1.4]{ks:regular}). We denote by $C_0(G^{(0)}, L^1(G,\lambda))$ the space of continuous sections vanishing at infinity endowed with the norm $\|f\|=\sup_{x\in G^{(0)}}\|f_{|x}\|_x$. Since it is complete and has ${\mathcal C}_c(G)$ as a dense subspace, the Banach spaces ${\mathcal E}$ and $C_0(G^{(0)}, L^1(G,\lambda))$ are identical. We need a description of the dual Banach space ${\mathcal E}^*$. This description could be derived from the appendix of \cite{gie:bundles}, where the general case of an upper semi-continuous Banach bundle $p:E\ra X$ is studied. We prefer to give a direct proof adapting to the non-Hausdorff case the results of Chapter 1 of \cite{adr:amenable}. As in Section 1.1 of \cite{adr:amenable}, we consider two locally compact (but not necessarily Hausdorff) spaces $X,Y$, a surjective continuous map $\pi:Y\ra X$ and a family 
$\alpha = \lbrace\alpha^x : x \in X\rbrace$
of positive Radon measures $\alpha^x$  on $\pi^{-1}(x)$ of full support such that
for every $f\in {\mathcal C}_c(Y)$, the function $\alpha(f) : x \mapsto
\int f d\alpha^x$ belongs to ${\mathcal C}_c(X)$. We call $\alpha$ a full continuous $\pi$-system. The following proposition extends \cite[Proposition 1.1.5]{adr:amenable} to the case when the space $Y$ is not necessarily Hausdorff. The proof is by reduction to the Hausdorff case

\begin{proposition}\label{dual} Let $\pi:Y\ra X$ and $\alpha$ be as above. We assume that $Y$ is $\sigma$-compact, locally compact but not necessarily Hausdorff and that $X$ and the fibers $\pi^{-1}(x)$ are Hausdorff. We define $C_0(X, L^1(Y,\alpha))$ as the completion of ${\mathcal C}_c(Y)$ for the norm $\|f\|=\sup_X\int |f| d\alpha^x$. Then the elements of its dual space  are complex Borel Radon measures on $Y$ of the form $\nu=\varphi(\mu\circ\alpha)$ where $\mu$ is a finite positive measure on $X$ and $\varphi\in L^\infty(Y,\mu\circ\alpha)$. The norm of $\nu$ is given by 
$$\|\nu\|=\inf \|\mu\|_1\|\varphi\|_\infty ,$$
 where the infimum is taken over all the representations $\nu=\varphi(\mu\circ\alpha)$. 
\end{proposition}
\begin{proof} Recall from \cite{ks:regular,mw:equivalence} that for for $f\in {\mathcal C}_c(Y)$, we do not have necessarily $|f|\in {\mathcal C}_c(Y)$. However $|f|$ is a Borel function and the function $x\mapsto \int |f|d\alpha^x$ is upper semi-continuous (\cite[Lemma 1.4]{ks:regular}). Just as in \cite{ks:regular}, we fix a cover of $Y$ ${\mathcal U}=(U_i)_{i\in I}$ by open Hausdorff subets $U_i$ and form the disjoint union $Y_{\mathcal U}=\sqcup_{i\in I}U_i$, which is a locally compact Hausdorff space. The identification map $\pi_{\mathcal U}: Y_{\mathcal U}\ra Y$ is a local homeomorphism. The system of counting measures along the fibers of $\pi_{\mathcal U}$ is a full continuous $\pi_{\mathcal U}$-system $\beta$ in the above sense. The corresponding map $\beta: C_c(Y_{\mathcal U})\ra {\mathcal C}_c(Y)$ satisfies $\|\beta(F)\|\le\sup_X\int |F| d(\alpha\circ\beta)^x$ where $(\alpha\circ \beta)^x=\int \beta^yd\alpha^x(y)$. Therefore, it extends to a norm-decreasing map $\beta: C_0(X, L^1(Y_{\mathcal U},\alpha\circ \beta))\ra C_0(X, L^1(Y,\alpha))$. Let $\phi$ be a continuous linear form on $C_0(X, L^1(Y,\alpha))$. Then $\phi\circ\beta$ is a continuous linear form on $C_0(X, L^1(Y_{\mathcal U},\alpha\circ \beta))$.   As in the Hausdorff case, the restriction of $\phi$ to ${\mathcal C}_c(Y)$ is a Radon measure; we denote by $\nu$ the associated Borel Radon measure on $Y$. The Borel Radon measure defined by the restriction of $\phi\circ\beta$ to $C_c(Y_{\mathcal U})$ is $\nu\circ\beta$. Since $Y_{\mathcal U}$ is Hausdorff, we can apply \cite[Proposition 1.1.5]{adr:amenable} to conclude that $\nu\circ\beta$ is $(\alpha\circ\beta)$-bounded, which means the existence of a finite positive measure $\mu$ on $X$ such that $|\nu\circ\beta|\le \mu\circ\alpha\circ\beta$ and $\|\mu\|_1\le \|\phi\circ\beta\|$. Since $|\nu\circ\beta|=|\nu|\circ\beta$ and every bounded Borel function $f$ on $Y$ with support contained in a compact subset can be written as $\beta(F)$ where $F$ is a bounded Borel function on $Y_{\mathcal U}$ with support contained in a compact subset, we obtain $|\nu|\le \mu\circ\alpha$. Moreover $\|\mu\|_1\le \|\phi\|$. Since $Y$ is $\sigma$-compact, $|\nu|$ is $\sigma$-finite. According to the Radon-Nikodym theorem, there exists $\varphi\in L^\infty (Y,\mu\circ\alpha)$ such that $\nu=\varphi(\mu\circ\alpha)$ and $\|\varphi\|_\infty\le 1$.
\end{proof}

Note that in this identification of the dual, positivity is respected: as mentioned earlier, a linear functional $\phi$ on ${\mathcal C}_c(Y)$ which is positive in the sense that $\phi(f)\ge 0$ for all $f\in {\mathcal C}_c(Y)^+$ defines a positive Borel Radon measure on $Y$.
\vskip 5mm          

It is well known that the convolution product of $f,g\in {\mathcal C}_c(G)$ defined by
$$(f*g)(\gamma_1)=\int f(\gamma)g(\gamma^{-1}\gamma_1) d\lambda^{r(\gamma_1)}(\gamma)$$
turns ${\mathcal C}_c(G)$ into an algebra and that $\|f*g\|\le\|f\|\|g\|$. Therefore, this product extends to $\mathcal E$ and turns it into a Banach algebra. Alternatively, by introducing , for $\gamma\in G$, the isometry
$$L(\gamma): L^1(G^{s(\gamma)},\lambda^{s(\gamma)})\ra L^1(G^{r(\gamma)},\lambda^{r(\gamma)})$$
defined by $L(\gamma)g_{s(\gamma)}(\gamma_1)=g_{s(\gamma)}(\gamma^{-1}\gamma_1)$, we may write the convolution product as a left action of $C_c(G)$ on $C_0(G^{(0)}, L^1(G,\lambda))$:
$$(L(f)g)_{|x}:=(f*g)_{|x}=\int f(\gamma)[L(\gamma)g_{|s(\gamma)}] d\lambda^x(\gamma)$$

For shorthand, we use the following notation: 
given $f\in {\mathcal C}_c(G)$, we define for $(\gamma,\gamma_1)\in G_{r,r}^{(2)}:=\{(\gamma,\gamma_1)\in G\times G: r(\gamma)=r(\gamma_1)\}$:
$$f'(\gamma,\gamma_1)=f(\gamma^{-1}\gamma_1)-f(\gamma_1).$$
Alternatively, we may view $f'$ as a section of the pull-back bundle $r^*L^1(G,\lambda)$:
$$f'_{|\gamma}=L(\gamma)f_{|s(\gamma)}-f_{|r(\gamma)}.$$
\vskip 5mm
 Given $f\in {\mathcal C}_c(G)$ and $m\in{\mathcal E}^{**}$, we define $f*m\in{\mathcal E}^{**}$ by bitransposition.

 \begin{definition}\label{mean}\cite[Definition 3.3.4]{adr:amenable} Let $(G,\lambda)$ be a locally compact groupoid with a continuous Haar system. A {\it topological invariant mean} is an element $m\in {\mathcal E}^{**}$, where ${\mathcal E}$ is the Banach space $C_0(G^{(0)},L^1(G,\lambda))$, such that
\begin{enumerate}
\item $\|m\|\le 1$ and $\nu\ge 0\Rightarrow m(\nu)\ge 0$;
\item for any probability measure $\mu$ on $G^{(0)}$, $m(\mu\circ\lambda)=1$;
\item for any $f\in C_c(G)$, we have $f*m=(\lambda(f)\circ r)m$.
\end{enumerate}
\end{definition}

Let us give two lemmas before stating and proving the main theorem.

We introduce the convex set
$$\Lambda^+_1={\mathcal C}_c(G)^+_1=\{f\in {\mathcal C}_c(G): f\ge 0,\quad \forall x\in G^{(0)}\quad \int f d\lambda^x\le 1\}.$$
We recall  the following result from \cite{adr:amenable}, where
$j:{\mathcal E}\ra {\mathcal E}^{**}$ denotes the canonical embedding the bidual. 

\begin{lemma}\cite[Lemma 1.2.7]{adr:amenable}
The image by $j$ of $\Lambda^+_1$ is dense in the positive part of the unit ball of ${\mathcal E}^{**}$ with respect to the weak*-topology.
\end{lemma}

We shall use two basic results \cite{buc:bc} about the strict topology of the multiplier algebra $C_b(X)$ of the commutative C*-algebra $C_0(X)$ without a unit. Although we only need the commutative case, it is as well to give the second result for an arbitrary C*-algebra. 

\begin{lemma}\cite[Theorem 1]{buc:bc}\label{buck1} Let $X$ be a locally compact Hausdorff space. The strict topology on the space $C_b(X)$ of bounded continuous functions on $X$ agrees on norm-bounded subsets of $C_b(X)$ with the the topology of uniform convergence on compact sets.
\end{lemma}

\begin{lemma}\cite[Theorem 2]{buc:bc}\label{buck2} Let $A$ be a C*-algebra. The inclusion map $i$ of $A$ into its multiplier algebra $M(A)$ identifies the dual of $M(A)$ equipped with the strict topology and the dual of $A$.
\end{lemma}

\begin{proof} (due to C. Anantharaman) The restriction map $i^*: M(A)^*_{\rm strict}\ra A^*$ is well defined because $i$ is continuous and it is injective because $A$ is dense in $M(A)_{\rm strict}$.
Its surjectivity is immediate from Cohen's factorization theorem: given $\varphi\in A^*$, there exist $\psi\in A^*$ and $a\in A$ such that $\varphi=\psi a$. Therefore, we can define the extension $\tilde\varphi$ by $\tilde\varphi(T)=\psi(aT)$ for $T\in M(A)$.
\end{proof}

We can now state and prove our main theorem.

\begin{theorem} \label {main} Let $(G,\lambda)$ be a $\sigma$-compact locally compact groupoid with Haar system. The following conditions are equivalent:
\begin{enumerate}
\item there exists a Borel approximate invariant mean;
\item there exists a topological invariant mean;
\item there exists a  topological approximate invariant mean.
\end{enumerate}
\end{theorem}

 \par
\mk
\begin{proof}
The proof is constructed along the same lines as in the classical case of a locally compact group (see \cite[Theorem G.3.1]{bhv:T}). The structure of the proof is $({\rm i})\Rightarrow ({\rm ii})\Rightarrow ({\rm iii})\Rightarrow ({\rm i})$. The last implication is trivial.
\vskip 3mm
$({\rm i})\Rightarrow ({\rm ii})$
 First note that any Borel function $g$ on $G$ such that $\int|g|d\lambda^x$ is bounded defines a bounded linear form $m_g$ on ${\mathcal E}^*$ according to the formula
$$\langle m_g,\varphi(\mu\circ\lambda)\rangle =\int g\varphi d(\mu\circ\lambda)$$
where $\mu$ is a bounded positive measure on $X$ and $\varphi\in L^\infty(G,\mu\circ\alpha)$ as in \propref{dual}. 
Indeed, according to Fubini's theorem, the integral is well-defined and depends only on the measure
 $\nu=\varphi(\mu\circ\lambda)$. Moreover, $\|m_g\|=\sup_x\int|g|d\lambda^x$. Let $(g_n)$ be a Borel approximate invariant mean. We have $\|m_{g_n}\|\le 1$. Let $m$ be a cluster point of the sequence $(m_n=m_{g_n})$ in ${\mathcal E}^{**}$ endowed with the weak* topology. I claim that $m$ is a topological invariant mean. Condition $({\rm i})$ of \defnref{mean} clearly holds. Let us check $({\rm ii})$. We have
$$m(\mu\circ\lambda)=\lim_n\int(\int g_nd\lambda^x)d\mu(x)=1$$
by Lebesgue dominated convergence theorem. Let us check $({\rm iii})$. Let $f\in {\mathcal C}_c(G)$. For $\nu=\varphi(\mu\circ\lambda)$ in ${\mathcal E}^*$, 
$$\begin{array}{ccl}
\langle f*m_n-(\lambda(f)\circ r)m_n,\nu\rangle &=&\langle \nu,f*g_n-(\lambda(f)\circ r )g_n\rangle \\
&=&\int f(\gamma)\big(\int\varphi(\gamma_1)g'_n(\gamma,\gamma_1)d\lambda^{r(\gamma)}(\gamma_1)\big) d(\mu\circ\lambda)(\gamma)
\end{array}$$
The integrand goes to 0 pointwise and is majorized by the integrable function $2\|\varphi\|_\infty |f|$. Therefore, this quantity goes to zero, which gives $({\rm iii})$. 
\vskip 3mm
$({\rm ii})\Rightarrow ({\rm iii})$ Let us denote by $f\mapsto m_f$ the canonical embedding of ${\mathcal E}$ into ${\mathcal E}^{**}$. Let $m$ be a topological invariant mean. Since the image of  $\Lambda^+_1$  is weak* dense in the positive part of the unit ball of ${\mathcal E}^{**}$, there exists a net $(g_i)$ in $\Lambda^+_1$ such that $(m_i=m_{g_i})$ tends to $m$ in the weak* topology. By construction, the net $(g_i)$ satisfies:
\begin{equation}
\label{eq:1} 
\forall x\in G^{(0)}\quad g_i\ge 0\quad{\rm and}\quad \int g_i d\lambda^x\le 1
\end{equation}

 It also satisfies:
\begin{equation}
\label{eq:2}  
\lambda(g_i)\to 1\quad \hbox{in the topology}\quad  \sigma(C_b(G^{(0)}), C_0(G^{(0)})^*)
\end{equation}
 Indeed, let $\mu$ be a probability measure $\mu$ on $G^{(0)}$. Then $\int \lambda(g_i)d\mu=m_i(\mu\circ\lambda)$ goes to $m(\mu\circ\lambda)=1$.
\\
Finally, let us show that the net $(g_i)$ satisfies
\begin{equation}
\label{eq:3} 
\quad\forall f\in {\mathcal C}_c(G),\quad f*g_i-(\lambda(f)\circ r) g_i\to 0\quad \hbox{in}\quad  \sigma({\mathcal E}, {\mathcal E}^*)
\end{equation} 
Indeed, let $\nu\in {\mathcal E}^*$. Then
$$\langle f*g_i-(\lambda(f)\circ r) g_i, \nu\rangle =\langle f*m_i-(\lambda(f)\circ r) m_i, \nu\rangle $$
tends to $\langle f*m-(\lambda(f)\circ r)m,\nu\rangle =0$.
\\
We endow ${\mathcal E}_0 := C_b(G^{(0)})$ with the strict topology and for each $f\in {\mathcal C}_c(G)$, we define ${\mathcal E}_f :={\mathcal E}$ and equip it with the norm topology. We equip the product space
$${\mathcal F}={\mathcal E}_0\times \Pi_{f\in {\mathcal C}_c(G)}{\mathcal E}_f$$
with the product topology. Then ${\mathcal F}$ is a locally convex space. We also consider the product space
$${\mathcal F}_w={\mathcal E}_{0,w}\times \Pi_{f\in {\mathcal C}_c(G)}{\mathcal E}_{f,w}$$
where ${\mathcal E}_{0,w}$ is equipped with the topology $\sigma(C_b(G^{(0)}), C_0(G^{(0)})^*)$ and ${\mathcal E}_{f,w}:= {\mathcal E}_{w}$ is equipped with the weak topology. Consider the following convex subset of $\mathcal F$:
$$C=\{(\lambda(g), (f*g-(\lambda(f)\circ r)g)_{f\in {\mathcal C}_c(G)}), g\in \Lambda^+_1\}.$$
Properties \eqref{eq:1},\eqref{eq:2} and \eqref{eq:3} say that the element $(1, (0)_{f\in {\mathcal C}_c(G)})$ belongs to the closure of $C$ in ${\mathcal F}_w$. According to \lemref{buck2}, the locally convex spaces ${\mathcal E}_0$ and ${\mathcal E}_{0,w}$ have the same continuous linear functionals. This also holds classically for the spaces ${\mathcal E}$ and ${\mathcal E}_w$. This remains true for the product spaces ${\mathcal F}$ and ${\mathcal F}_w$. Therefore, according to a corollary of the Hahn-Banach theorem, the closure of the convex set $C$ is the same in both spaces. This implies the existence of a net, which we still call $(g_i)$, in $\Lambda^+_1$ such that $\lambda(g_i)$ tends to 1 in the strict topology of  $C_b(G^{(0)})$ and such that for every $f\in {\mathcal C}_c(G)$, $\sup_x\int |f*g_i-\lambda(f)\circ r)g_i|d\lambda^x$ goes to 0. Since, as we have seen in \lemref{buck1}, the strict topology coincides with the topology of uniform convergence on compact sets, the first condition may be written as:\begin{equation}
\label{eq:4}  
\int g_i d\lambda^x\to 1\quad \hbox{uniformly on compact subsets of}\,\,  G^{(0)}
\end{equation}
We may write the second condition as
\begin{equation}
\label{eq:5} 
\forall f\in {\mathcal C}_c(G)\quad \|\int f(\gamma_1){g'_i}_{|\gamma_1}d\lambda^x(\gamma_1)\|_x\to 0\quad \hbox{uniformly on}\,\,  G^{(0)}
\end{equation}
where $\|.\|_x$ is the norm of $L^1(G^x,\lambda^x)$. Let us show that \eqref{eq:5} implies an apparently stronger condition: 
\begin{equation}
\label{eq:6}  
\forall F\in {\mathcal C}_c(G_{r,r}^{(2)})\quad\|\int F(\gamma,\gamma_1){g'_i}_{|\gamma_1}d\lambda^{r(\gamma)}(\gamma_1)\|_{r(\gamma)}\to 0\quad \hbox{uniformly on}\,\,  G
\end{equation}
This is clear when $F$ is of the form $f_1\otimes f_2$, where $f_1,f_2\in {\mathcal C}_c(G)$. Let $F\in C_c(U_1*U_2)$, where   $U_1,U_2$ are relatively compact open Hausdorff subsets of $G$ and $U_1*U_2=(U_1\times U_2)\cap G_{r,r}^{(2)}$. Given $\epsilon>0$, according to the Stone-Weierstrass theorem, there exists a function $\underline F \in C_c(U_1*U_2)$ of the form $\sum_{k=1}^n f_{1,k}\otimes f_{2,k}$, where $f_{i,k}\in C_c(U_i)$ such that, for all $(\gamma,\gamma_1)\in U_1*U_2$, $|F(\gamma,\gamma_1)-\underline F(\gamma,\gamma_1)|\le\epsilon$. We have for all $\gamma\in U_1$:
$$\begin{array}{cl}
\|\int (F(\gamma,\gamma_1)-{\underline F}(\gamma,\gamma_1)){g'_i}_{|\gamma_1}d\lambda^{r(\gamma)}(\gamma_1)\|_{r(\gamma)}&\le \epsilon\int_{U_2} \|{g'_i}_{|\gamma_1}\|_{r(\gamma)}d\lambda^{r(\gamma)}(\gamma_1)\\
&\le 2\epsilon \lambda^{r(\gamma)}(U_2)\\
&\le 2M\epsilon
\end{array}$$
where $M=\sup_{x\in r(U_1)}\lambda^x(U_2)$ is finite because of the continuity of the Haar system and the relative compactness of $U_2$ and $r(U_1)$. This inequality holds for all $\gamma\in G$ when we replace $F$ [resp. ${\underline F}$] by its extension $\tilde F$ [resp. $\tilde{\underline F}$]  by 0 outside $U_1*U_2$. Combining this inequality with the convergence result for $\tilde{\underline F}$, we obtain the desired convergence for $\tilde F$. Since an arbitrary element of ${\mathcal C}_c(G_{r,r}^{(2)})$ is a linear combination of such functions $\tilde F$, we obtain \eqref{eq:6}.\\
Next, we would like to show the following property:
\begin{equation}
\label{eq:7}
\forall f\in {\mathcal C}_c(G)\quad\|\int f(\gamma^{-1}\gamma_1){g'_i}_{|\gamma_1}d\lambda^{r(\gamma)}(\gamma_1)\|_{r(\gamma)}\to 0
\end{equation}
uniformly on compact subsets of $G$. This cannot be derived directly from \eqref{eq:6} because the function sending $(\gamma,\gamma_1)$ in $G_{r,r}^{(2)}$ to $f(\gamma^{-1}\gamma_1)$ does not have compact support. However, one can proceed as follows. Let $K$ be a compact subset of $G$. Since every element of $G$ has an open neighborhood contained in a compact set having a Hausdorff neighborhood, $K$ is contained in a finite union of compact subsets $K_1, K_2, \ldots, K_l$ having Hausdorff open neighborhoods $U_1, U_2, \ldots, U_l$. There exists for each $j=1,\ldots, l$ a function $h_j\in C_c(U_j)$ such that  $0\le h_j\le 1$ and $h_j(\gamma)=1$ for all $\gamma\in K_j$. Then $h=\sum_{j=1}^l \tilde h_j$ belongs to ${\mathcal C}_c(G)$ and we can define $F$ on $G_{r,r}^{(2)}$ by $F(\gamma,\gamma_1)=h(\gamma) f(\gamma^{-1}\gamma_1)$. It belongs to ${\mathcal C}_c(G_{r,r}^{(2)})$ because $h\otimes f$ does and the map $(\gamma,\gamma_1)\mapsto (\gamma, \gamma^{-1}\gamma_1)$ is a homeomorphism of $G_{r,r}^{(2)}$ onto itself. For $\gamma\in \bigcup_{j=1}^lK_j$, we have
$$\|\int f(\gamma^{-1}\gamma_1){g'_i}_{|\gamma_1}d\lambda^{r(\gamma)}(\gamma_1)\|_{r(\gamma)}\le \|\int F(\gamma,\gamma_1){g'_i}_{|\gamma_1}d\lambda^{r(\gamma)}(\gamma_1)\|_{r(\gamma)}$$
and by \eqref{eq:6}, the left hand side tends to 0 uniformly. This gives \eqref{eq:7}.
\vskip 2mm
Suppose that we are given a compact subset $L$ of $G^{(0)}$, a compact subset $K$ of $G$ and $\epsilon>0$. We are going to construct $g\in\Lambda^+_1$ such that
\begin{equation}
\label{eq:8}  
\forall x\in L,\quad 1-\|g_{|x}\|_x\le \epsilon
\end{equation}
\begin{equation}
\label{eq:9}  
\forall \gamma\in K,\quad \|g'_{|\gamma}\|_{r(\gamma)}\le \epsilon
\end{equation}
We choose $f\in {\mathcal C}_c(G)^+$ such that $\int f d\lambda^x=1$ for all $x$ in the compact subset $L'= L\cup s(K)\cup r(K)$.\\
According to \eqref{eq:4}, there exists $i_0$ such that for $i\ge i_0$,
$$\|{g_i}_{|x}\|_x\ge 1- \epsilon\qquad \forall x\in s({\rm supp}f).$$
According to \eqref{eq:7},  there exists $i\ge i_0$ such that for all $\gamma\in K$,
$$\max(\|\int f(\gamma^{-1}\gamma_1){g'_i}_{|\gamma_1}d\lambda^{r(\gamma)}(\gamma_1)\|_{{r(\gamma)}}, \|\int f(\gamma_1){g'_i}_{|\gamma_1}d\lambda^{r(\gamma)}(\gamma_1)\|_{{r(\gamma)}})\le \epsilon/2$$

Pick such an $i$ and consider the function $g=f*g_i$.  Then $g\in\Lambda^+_1$ and for all $x\in L$,
$$\begin{array}{ccl}
\|g_{|x}\|_x&=&\int g(\gamma_1)d\lambda^x(\gamma_1)\\
&=&\int\int f(\gamma)g_i(\gamma^{-1}\gamma_1)d\lambda^{r(\gamma_1)}(\gamma)d\lambda^x(\gamma_1)\\
&=&\int f(\gamma)\int g_i(\gamma^{-1}\gamma_1)d\lambda^{r(\gamma)}(\gamma_1)d\lambda^x(\gamma)\\
&=&\int f(\gamma)\int g_i(\gamma_1)d\lambda^{s(\gamma)}(\gamma_1)d\lambda^x(\gamma)\\
&\ge& (1-\epsilon)\int f(\gamma)d\lambda^x(\gamma)=1-\epsilon\end{array}$$
Thus, \eqref{eq:8} is realized. On the other hand, we have the following equality: for all $\gamma\in G_{L'}^{L'}$,
$$g'_{|\gamma}=\int f(\gamma^{-1}\gamma_1){g_i}'_{|\gamma_1}d\lambda^{r(\gamma)}(\gamma_1)-\int f(\gamma_1){g_i}'_{|\gamma_1}d\lambda^{r(\gamma)}(\gamma_1).$$
Thus, if $\gamma\in K$, we have $\|g'_{|\gamma}\|_{r(\gamma)}\le \epsilon$. Therefore \eqref{eq:9} is also realized.
\\
Since $G$ [resp. $G^{(0)}$] is locally compact and $\sigma$-compact, there exists an increasing sequence $(K_n)_{n\in\N}$ [resp. $(L_n)_{n\in\N}$] of compact subsets of $G$ [resp. $G^{(0)}$] such that $G=\bigcup_{n\in\N}K_n$ [resp. $G^{(0)}=\bigcup_{n\in\N}L_n$]. We also choose a sequence $(\epsilon_n)_{n\in\N}$ decreasing to 0. For every $n\in\N$, there exists $g_n$ in ${\mathcal C}_c(G)_1^+$ such that $1-\|{g_n}_{|x}\|_x\le \epsilon_n$ for all $x\in L_n$ and $\|{g_n}'_{|\gamma}\|_{r(\gamma)}\le \epsilon_n$ for all $\gamma\in K_n$. Then $(g_n)_{n\in\N}$ is a topological approximate invariant mean.
\end{proof}


Above theorem can be rephrased as:

\begin{corollary} Let $(G,\lambda)$ be a $\sigma$-compact locally compact groupoid with Haar system. Then $G$ is topologically amenable if and only if it is Borel amenable.
\end{corollary}

\section{Examples and applications}

\subsection{Applications}

Applications of amenability to operator algebras are well-known. The main results are that the full and the reduced C*-algebras of a locally compact groupoid $G$ endowed with a Haar system, denoted respectively $C^*(G)$ and $C^*_r(G)$,  coincide when the groupoid is amenable and that $C^*(G)$ is nuclear. In \cite[Chapter 4]{adr:amenable}, these results are established for second countable Hausdorff locally compact groupoids and use only Borel amenability (in fact, the weaker condition of measurewise amenability suffices); they rely on a theorem of disintegration of representations. They are valid along with their proofs for non-Hausdorff groupoids as well. On the other hand, topological amenability of $G$ provides an alternative proof of these results, at least in the Hausdorff case. Indeed it can then be expressed as the existence of a sequence $(h_n)$ of continuous positive type functions with compact support on $G$, with ${h_n}_{|G^{(0)}}\le 1$, which converges to 1 uniformly on compact subsets (see \cite[Proposition 2.2.13]{adr:amenable}). Since pointwise multiplication by a bounded continuous positive type function $h$ defines a completely positive linear map $m_h$ on $C^*(G)$ [resp. $C^*_r(G)$] to itself (see \cite[Theorem 4.1]{rw:Fourier}), one gets a sequence $(m_{h_n})$ of completely positive linear maps completely bounded by 1 converging to the identity in the point-norm topology. this provides an approximation property which implies both the equality of the full and the reduced norms and the nuclearity of $C^*(G)$ (see \cite[Th\'eor\`eme 4.9]{ad:syst} and \cite[Theorem 5.6.18]{bo:fda}). As shown by J.-L. Tu in \cite{tu:amenable}, topological amenability as expressed in \defnref{topological approximate density} gives directly the fact that an amenable locally compact $\sigma$-compact groupoid $G$ with a Haar system acts properly on a continuous field of affine euclidean spaces. This has two important consequences: $G$ satisfies the Baum-Connes conjecture \cite[Th\'eor\`eme 9.3]{tu:amenable} and $C^*(G)$ satisfies the Universal Coefficient Theorem \cite[Proposition 10.7]{tu:amenable}.
\subsection{Orbit equivalence}
One of the main properties of Borel [resp. topological] amenability is its invariance under Borel [resp. topological] equivalence of groupoids. The definition of Borel equivalence is given in \cite[Definition A.1.11]{adr:amenable}. Invariance under topological equivalence is established in \cite[Theorem 2.2.17]{adr:amenable}. The proof is easily adapted to the Borel case.
In their work on Cantor minimal systems, Giordano, Putnam and Skau have introduced a notion of topological orbit equivalence which we recall. Let us denote by $(X,R)$ an equivalence relation $R$ on a set $X$; we view an equivalence relation as a groupoid $R\subset X\times X$. We assume that $X$ is a topological space and that $R$ is a Borel subset of $X\times X$. Equivalence relations $(X,R)$ and $(\underline X,\underline R)$ are said to be topologically orbit equivalent if there exists a homeomorphism $\varphi^{(0)}: X\ra{\underline X}$ such that $\varphi(R)=\underline R$, where $\varphi(x,y)=(\varphi^{(0)}(x),\varphi^{(0)}(y))$. Since $\varphi$ is an isomorphism of Borel groupoids, topological orbit equivalence preserves Borel amenability. This remains true for Kakutani equivalence as defined in \cite{gps:orbit equivalence}.  It turns out that the equivalence relation $(X,R)$ associated with a Cantor minimal system has several topologies  which turn $R$ into an \'etale locally compact groupoid. However, the underlying Borel structure is necessarily the Borel structure inherited from $X\times X$. Therefore, if one these \'etale groupoid is topologically amenable, then according to \thmref{main}, so are the others.  In particular, an equivalence relation which is {\it affable}, i.e. topologically orbit equivalent to an AF equivalence relation, is necessarily amenable. Since for \'etale Hausdorff locally compact groupoids, topological amenability is equivalent to the nuclearity of the (reduced) C*-algebra (\cite[Theorem 5.6.18]{bo:fda}), either all the associated C*-algebras are nuclear or none is nuclear.

\subsection{Singly generated dynamical systems}
We define a {\it singly generated dynamical system} (SGDS) as in \cite[Definition 2.3]{ren:cuntz}. It is  a
pair $(X,T)$ where $X$ is a topological space and
$T$ is a local homeomorphism from an open subset $\operatorname{dom}(T)$ of $X$ onto an open subset $\operatorname{ran}(T)$ of $X$. They are quite common dynamical systems, which appear either directly (e.g. one-sided subshifts of finite type) or as canonical extensions (see \cite{tho:kms, ds:interval}). The case where both the domain and the range of $T$ are strictly included in $X$ is found in graph and higher-rank graphs algebras (see for example \cite{rswy:higher}). The {\it semi-direct product groupoid} of a SGDS $(X,T)$ is defined (\cite[Definition 2.4]{ren:cuntz}) as:
$$G(X,T)=\{(x,m-n,y): m,n\in {\bf N},  x\in \operatorname{dom}(T^m), y\in \operatorname{dom}(T^n), T^mx=T^ny\}$$
with the groupoid structure induced by the product structure
of the trivial groupoid $X\times X$ and of the group ${\bf Z}$ and the topology defined
by the basic open sets
$${\mathcal U}(U;m,n;V)=\{(x,m-n,y): (x,y)\in U\times V,\ T^m(x)=T^n(y)\}$$ 
where $U$  [resp.
$V$] is an open subset of the domain of $T^m$ [resp. $T^n$] on
which $T^m$ [resp. $T^n$] is injective.

\begin{proposition}\label{SGDS}(cf. \cite[Proposition 2.9.(i)]{ren:cuntz}) Let $(X,T)$ be a SGDS where $X$ is locally compact, second countable and Hausdorff. Then its semi-direct groupoid $G(X,T)$ is topologically amenable.\end{proposition}

\begin{proof} As observed by D. Williams, the proof given in \cite[Proposition 2.4]{ren:cuntz} is not valid when $\operatorname{ran}(T)$ is strictly contained in $X$. Using the same idea and taking advantage of our main theorem, we will establish topological amenability in the general case. There are two groupoids associated with the fundamental cocycle $c:G(X,T)\ra\Z$, given by $c(x,k,y)=k$. Namely, its kernel $c^{-1}(0)$, which is a closed subgroupoid of $G=G(X,T)$ and the skew-product $G(c)$. We use here the convention of \cite[Definition I.1.6]{ren:approach}). Thus  $G(c)$ is defined as $G\times\Z$ with unit space is $X\times\Z$. The range and source maps are respectively given by $r(\gamma,a)=(r(\gamma),a)$ and $s(\gamma, a)=(s(\gamma),a+c(\gamma)$. Two elements $(\gamma,a)$ and $(\gamma',b)$ of $G(c)$ are composable if and only if $\gamma$ and $\gamma'$ are composable and $b=a+c(\gamma)$; then their product is $(\gamma, a)(\gamma',b)=(\gamma\gamma', a)$. Let us introduce the subspace $Y=\{(s(\gamma), c(\gamma)): \gamma\in G\}$ of the unit space $X\times\Z$ of $G(c)$. It is open, because for all $k\in\Z$, $c^{-1}(k)$ is open and the source map is open, and invariant under $G(c)$. Let $H$ be the reduction of $G(c)$ to $Y$. Then $H$ and $c^{-1}(0)$ are topologically equivalent. The equivalence is implemented by $G$, endowed with the natural left action of $c^{-1}(0)$ and the right action of $H$ given by $z(\gamma, c(z))=z\gamma$.  The kernel $c^{-1}(0)$  is Borel amenable because it is an increasing union of proper Borel groupoids. According to \cite [Theorem 2.2.17]{adr:amenable} (or rather, a Borel version of it), $H$ is also Borel amenable. Then we proceed as in \cite [Proposition II.3.8]{ren:approach} to show that $G(X,T)$ is Borel amenable.  The space $Y$ is endowed with the left action of $G$ defined by the first projection $p: Y\ra X$ as anchor map and the formula $\gamma(s(\gamma), a+c(\gamma))=(r(\gamma), a)$. The $G$-map $p$ is Borel amenable in the sense of \cite[Definition 2.2.2]{adr:amenable}, or rather a Borel version of this definition. This means the existence of a sequence $(\mu_j)_{j\in\N}$, where each $\mu_j$ is a family $(\mu_j^x)_{x\in X}$ of probability measures $\mu_j^x$ on $c(G_x)$ such that for all bounded Borel functions $f$ on $Y$, $x\mapsto \int f(x,a)d\mu_j^x(a)$ is Borel and for all $\gamma\in G$, 
$\|\gamma \mu_j^{s(\gamma)}-\mu_j^{r(\gamma)}\|_1$ tends to $0$. Explicitly, we can take
$$\mu_j^x=\frac{1}{|c(G_x)\cap\{-j,\ldots, j\}|}\sum_{a\in c(G_x)\cap\{-j,\ldots, j\}} \delta_a$$
because the subsets $c(G_x)$ are either $\Z$, semi-infinite intervals $\{a(x), \dots\}$ or finite intervals $\{a(x),\ldots, b(x)\}$. The Borel amenability of $H$ is exactly the Borel amenability of the $G$-map $r:H\ra Y$, where the action of $G$ on $H$ is given by $\gamma(\gamma', a+c(\gamma))=(\gamma\gamma',a)$. A Borel version of \cite[Proposition 2.2.4]{adr:amenable} gives the Borel amenability of the $G$-map $p\circ r:H\ra X$. Since $H$ is a principal $G$-space, a Borel version of \cite[Corollary 2.2.10]{adr:amenable} gives the Borel amenability of $G$. One can also prove the result directly: let $(m_i^{(x,a)})_{(x,a)\in Y}$ be a Borel approximate invariant mean for $H$. By identifying $H^{(x,a)}$ and $G^x$, one can define the measures
$$m_{ij}^x=\int m_i^{(x,a)}d\mu_j^x(a).$$
It is routine to check that $(m_{ij})$ is a Borel approximate invariant mean for $G(X,T)$. 
\end{proof}

\begin{remark}
As pointed out in \cite{rswy:higher}, there is an alternative proof of this result: it is an immediate consequence of \cite[Proposition 9.3]{spi:gpd}. The above proof is more elementary in the sense that it does not use C*-algebras.
\end{remark}

\subsection{Groupoid bundles}

\begin{definition} A locally compact groupoid $G$ over $G^{(0)}=X$ is a {\it groupoid bundle} over a locally compact Hausdorff space $T$ if there exists a continuous open surjection $p:X\ra T$ which is invariant in the sense that $p\circ r=p\circ s$. \end{definition}

Then for all $t\in T$, $X(t)=p^{-1}(t)$ is a closed invariant subset.  We define here the groupoid $G(t)$ as the reduction $G_{|X(t)}$.\\

 Let $(G,\lambda)$ be a $\sigma$-compact locally compact groupoid with Haar system. We recall that $G$ is topologically amenable if and only if
given a compact subset $L$  of $G^{(0)}$, a compact subset $K$ of $G$ and $\epsilon>0$, there exists $g\in {\mathcal C}_c(G)^+$ such that
\begin{enumerate}
\item $\|g\|_x\le 1$ for all $x\in G^{(0)}$;
\item $\|g\|_x\ge 1-\epsilon$ for all $x\in L$;
\item $\|g'_{|\gamma}\|_{r(\gamma)}\le \epsilon$ for all $\gamma\in K$.
\end{enumerate}

We shall say that a function $g$ is {\it $(L,K,\epsilon)$-invariant} if it satisfies $({\rm i}), ({\rm ii})$ and $({\rm iii})$.\\

The proof of the theorem below relies essentially on the Tietze extension theorem. The version of this theorem for non-Hausdorff locally compact spaces which we need can be found in \cite{tu:non-hausdorff}. It is deduced from the classical Hausdorff case by using the same technique as in \propref{dual}.

\begin{lemma}\label{Tietze}\cite[Lemma 4.5]{tu:non-hausdorff} Let $Z$ be a closed subset of a non-Hausdorff locally space $Y$. Given $g\in {\mathcal C}_c(Z)^+$, there exists $f\in {\mathcal C}_c(Y)^+$ such that $g=f_{|Z}$.
\end{lemma}
\vskip5mm

\begin{theorem}
Let $G$ be a $\sigma$-compact locally compact groupoid with continuous Haar system. Assume that $p:G\ra T$ is a groupoid bundle over a locally compact Hausdorff space $T$. Then \tfae
\begin{enumerate}
\item $G$ is topologically amenable;
\item for all $t\in T$, $G(t)$ is topologically amenable.
\end{enumerate}
\end{theorem}

\begin{proof} $({\rm i})\Rightarrow ({\rm ii})$  Let $(m_n)$ be a topological approximate invariant mean for $G$. Fix $t\in T$. Then $(m^x_n)_{x\in X(t)}$ is a topological approximate invariant mean for $G(t)$.\\

$({\rm ii})\Rightarrow ({\rm i})$ Let $L$ compact subset of $X=G^{(0)}$, $K$ compact subset of $G$ and $\epsilon>0$ be given. Fix $t\in p(L)$. Since $G(t)$ is amenable, there exists $g_{|t}\in {\mathcal C}_c(G(t))^+$ satisfying the $(L\cap X(t), K\cap G(t),\epsilon/2)$-condition. According to \lemref{Tietze}, there exists $g\in{\mathcal C}_c(G)^+$ which extends $g_{|t}$. Because of the continuity of $x\mapsto\int g d\lambda^x$, we may scale $g$ so that  $\int g d\lambda^x\le 1$ for all $x\in X$ and $\int g d\lambda^x\ge 1-\epsilon$ for all $x\in L$. Since the function $\gamma\mapsto \|g'_{|\gamma}\|_{r(\gamma)}$ is upper semi-continuous (as we have seen earlier, the bundle $E=L^1(G,\lambda)$ is upper semi-continuous and so is the pull-back bundle $r^*E$ above $G$), the set
 $U=\{\gamma\in G: \|g_{|\gamma}\|_{r(\gamma)}<\epsilon\}$ is open. Since $K\setminus U$ is compact, $p\circ r(K\setminus U)$ is also compact hence closed in $T$. Moreover, this closed set does not contain $t$. Its complement is an open neighborhood $V_t$ of $t$ such that $\|g_{|\gamma}\|_{r(\gamma)}< \epsilon$ for all $\gamma\in K\cap (p\circ r)^{-1}(V_t)$. In summary, for every $t\in p(L)$, there exists an open neighborhhood $V_t$ of $t$ and a function  $g_t \in {\mathcal C}_c(G)^+$ which is $(L,K\cap(p\circ r)^{-1}(V_t),\epsilon)$-invariant.
\\
By compactness of $p(L)$, we obtain a finite open cover  $(V_i)_{i=1,\ldots,n}$ and for each $i=1,\ldots, n$ an $(L,K\cap(p\circ r)^{-1}(V_i),\epsilon)$-invariant function $g_i\in {\mathcal C}_c(G)^+$. Let $(h_1,\ldots, h_n)$ be a partition of unity subordinate to the cover. Define $g=\sum_{i=1}^n (h_i\circ p\circ r)g_i$. Then $g$ belongs to $C_c(G)^+$. It satisfies $({\rm i})$ and $({\rm ii})$. Since
 $g'_{|\gamma}=\sum_{i=1}^n (h_i\circ p\circ r)(\gamma)g'_{i|\gamma}$, we have $\|g'_{|\gamma}\|_{r(\gamma)}\le\sum_{i=1}^n (h_i\circ p\circ r)(\gamma) \|g'_{i |\gamma}\|_{|r(\gamma)}$. If $\gamma \in K$, we have $\|g'_{i |\gamma}\|_{|r(\gamma)}\le \epsilon$ for all $i$ such that $p\circ r(\gamma)$ belongs to $V_i$, hence $\|g'_{|\gamma}\|_{r(\gamma)}\le\epsilon$. Therefore, $g$ is $(L,K,\epsilon)$-invariant. 

\end{proof}

\begin{remark}
It would be interesting to have a Borel version of this theorem: a decomposition provided by a Borel invariant map $p: G^{(0)}\ra T$ is closer to the usual ergodic decomposition for measured groupoids than our topological version.
\end{remark}

\subsection{F\o lner sets and growth conditions}

Since \cite[Section 3.2.c]{adr:amenable} deals with measured groupoids rather than  Borel or topological groupoids, the following complements may be useful.

\begin{definition}\label{Folner}
Let $(G,\lambda)$ be a Borel groupoid endowed with a Borel Haar system. A {\it F\o lner sequence} is a sequence $(F_n)$ of Borel subsets of $G$ such that
\begin{enumerate}
\item for all $n$ and for all $x\in G^{(0)}$, $0<\lambda^x(F_n)<\infty$;
\item for all $\gamma\in G$, $\displaystyle \lim_n\frac{\lambda^{r(\gamma)}(\gamma F_n\Delta F_n)}{\lambda^{s(\gamma)}(F_n)}=0$.
\end{enumerate}
\end{definition}

\begin{remark}\label{Folner bis}
Just as in the case of groups, condition $({\rm ii})$ can be replaced by  

	 $({\rm ii}\, {\rm bis})$ for all $\gamma\in G$, $\displaystyle \lim_n\frac{\lambda^{r(\gamma)}(\gamma F_n\cap F_n)}{\lambda^{s(\gamma)}(F_n)}=1$. \\
Note that both conditions $({\rm ii})$ and $({\rm ii}\, {\rm bis})$ imply that  $\displaystyle \lim_n\frac{\lambda^{r(\gamma)}(F_n)}{\lambda^{s(\gamma)}(F_n)}=1$ for all $\gamma\in G$. 
\end{remark}

\begin{lemma}\label{estimate} Let $(G,\lambda)$ be a Borel groupoid endowed with a Borel Haar system. Let $F$ be a Borel subset of $G$ such that for all $x\in G^{(0)}$, $0<\lambda^x(F)<\infty$. Then the normalized characteristic function $g_F$ defined by 
$g_F(\gamma)=\frac{{\bf 1}_{F}(\gamma)}{\lambda^{r(\gamma)}(F)}$
satisfies for all $\gamma\in G$ the inequality
$$\int| g_F(\gamma^{-1}\gamma_1)-g_F(\gamma_1)|d\lambda^{r(\gamma)}(\gamma_1)\le 2\frac{\lambda^{r(\gamma)}(\gamma F\Delta F)}{\lambda^{s(\gamma)}(F)}.$$
\end{lemma}

\begin{proof}
The proof is an easy computation:
$$g(\gamma^{-1}\gamma_1)-g(\gamma_1)=
\frac{\lambda^{r(\gamma)}(F){\bf 1}_{F}(\gamma^{-1}\gamma_1)-\lambda^{s(\gamma)}(F){\bf 1}_{F}(\gamma_1)}{\lambda^{r(\gamma)}(F)\lambda^{s(\gamma)}(F)}.$$
Adding and subtracting
$\lambda^{r(\gamma)}(F){\bf 1}_{F}(\gamma_1)$, one gets:
$$|g(\gamma^{-1}\gamma_1)-g(\gamma_1)|\le
\frac{\lambda^{r(\gamma)}(F){\bf 1}_{\gamma F\Delta  F}(\gamma_1)+|\lambda^{r(\gamma)}(F)-\lambda^{s(\gamma)}(F)|{\bf 1}_{F}(\gamma_1)}{\lambda^{r(\gamma)}(F)\lambda^{s(\gamma)}(F)}.$$
Moreover, one has
$$|\lambda^{r(\gamma)}(F)-\lambda^{s(\gamma)}(F)|=
|\int ({\bf 1}_{F}-{\bf 1}_{\gamma F})d\lambda^{r(\gamma)}|\le \lambda^{r(\gamma)}(\gamma F\Delta F).$$ 
Thus, by integrating over $\gamma_1$, one obtains the inequality
$$\int| g(\gamma^{-1}\gamma_1)-g(\gamma_1)|d\lambda^{r(\gamma)}(\gamma_1)\le 2\frac{\lambda^{r(\gamma)}(\gamma F\Delta F)}{\lambda^{s(\gamma)}(F)}.$$
\end{proof}

\begin{proposition}\label{Folner} Let $(G,\lambda)$ be a Borel groupoid endowed with a Borel Haar system. If there exists a F\o lner sequence, then $G$ is Borel amenable.
\end{proposition}

\begin{proof} We check that the sequence $(g_n)$ of normalized characteristic functions of the $F_n$'s satisfies \defnref{Borel approximate density}. It is clear that $g_n$ is a non-negative Borel function and that for all $x\in G^{(0)}$, $\int g_nd\lambda^x=1$. The approximate invariance is given by the above lemma.
\end{proof}

\begin{corollary}\label{Alembert} Let $(G,\lambda)$ be a Borel groupoid endowed with a Borel Haar system. Let  $(E_n)_{n\in\N}$ be an increasing  and exhausting sequence of Borel subsets of $G$ such that
\begin{enumerate}
\item for all $n\in\N$ and for all $x\in G^{(0)}$, $0<\lambda^x(E_n)<\infty$;
\item for all $m,n\in\N$, $E_mE_n\subset E_{m+n}$;
\item for all $x\in G^{(0)}$, $\displaystyle\frac{\lambda^x(E_{n+1})}{\lambda^x(E_n)}$ tends to 1.
\end{enumerate}
Then $(E_n)$ is a F\o lner sequence; therefore $G$ is Borel amenable.
\end{corollary}

\begin{proof} Assuming that $\gamma\in E_k$ and $k\le n$, we have:
$$\gamma E_{n-k}\subset\gamma E_n\cap E_n.$$
Using the relation $\lambda^{s(\gamma)}(E_{n-k})=\lambda^{r(\gamma)}(\gamma E_{n-k})$, we deduce the inequality
\begin{equation}
\label{eq:3.1}
\frac{\lambda^{s(\gamma)}(E_{n-k})}{\lambda^{s(\gamma)}(E_n)}\le\frac{\lambda^{r(\gamma)}(\gamma E_n\cap E_n)}{\lambda^{s(\gamma)}(E_n)}\le 1
\end{equation}
Our assumption $({\rm iii})$ implies that the left handside goes to 1 when $n$ goes to infinity. Therefore, $(E_n)$ satisfies condition $({\rm ii}\, {\rm bis})$ of \remref{Folner bis}.
\end{proof}

In our next result, we replace the d'Alembert ratio $\displaystyle\frac{\lambda^x(E_{n+1})}{\lambda^x(E_n)}$ by the Cauchy exponent
$\displaystyle\lambda^x(E_n)^{1/n}$. However, we shall need stronger hypotheses to obtain amenability. We first adapt to the Borel setting our definition of $(L,K,\epsilon)$-invariance.

\begin{definition}\label{(L,K,eps)} Let $(G,\lambda)$ be a Borel groupoid endowed with a Borel Haar system. Given a Borel subset $L$ of $G^{(0)}$, a Borel subset $K$ of $G$ and $\epsilon>0$, we say that a non-negative Borel function $g$ on $G$ is $(L,K,\epsilon)$-{\it invariant} if
\begin{enumerate}
\item $\int g d\lambda^x\le 1$ for all $x\in G^{(0)}$;
\item $\int g d\lambda^x=1$ for all $x\in L$;
\item $\int|g(\gamma^{-1}\gamma_1)-g(\gamma_1)|d\lambda^{r(\gamma)}(\gamma_1)\le \epsilon$ for all $\gamma\in K$.
\end{enumerate}
\end{definition}

\begin{lemma}\label{glueing} Let $(G,\lambda)$ be a Borel groupoid endowed with a Borel Haar system. Let $L$ be a Borel subset of $G^{(0)}$, $K$  a Borel subset of $G$  and $\epsilon>0$. Let $(L_i)_{i\in I}$ be a locally finite Borel cover of $L$. Suppose that for each $i\in I$, there exists a non-negative Borel function $g_i$ on $G$ which is $(L_i, K^{L_i},\epsilon)$-invariant. Then there exists a non-negative Borel function $g$ on $G$ which is $(L, K^L,\epsilon)$-invariant.
\end{lemma}
 
\begin{proof}
Let $(h_i)_{i\in I}$ be a Borel partition of unity subordinate to $(L_i)_{i\in I}$: $h_i$ is Borel, $0\le h_i\le 1$, $h_i(x)=0$ if $x\notin L_i$ and for all $x\in G^{(0)}$, $\sum_{i\in I} h_i(x)=1$. Then $g=\sum_{i\in I} (h_i\circ r)g_i$ is a well-defined, non-negative, Borel function on $G$. For all $x\in G^{(0)}$, $\int g d\lambda^x=\sum_I h_i(x)\int g_i d\lambda^x\le 1$. For all $x\in L$, $\int g d\lambda^x=\sum_I h_i(x)\int g_i d\lambda^x = 1$. If $\gamma\in K^L$, then
 $$\int |g(\gamma^{-1}\gamma_1)-g(\gamma_1)|d\lambda^{r(\gamma)}\le\sum_I h_i(r(\gamma))\int |g_i(\gamma^{-1}\gamma_1)-g_i(\gamma_1)|d\lambda^{r(\gamma)}(\gamma_1)\le \epsilon.$$ 
 \end{proof}

\begin{lemma}\label{aim} Let $(G,\lambda)$ be a Borel groupoid endowed with a Borel Haar system. If there exist $(L_n)$ increasing and exhausting sequence of Borel subsets of $G^{(0)}$, $(K_n)$ increasing and exhausting sequence of Borel subsets of $G$ and $(\epsilon_n)$ a sequence of positive numbers decreasing to $0$ and for each $n$ a $(L_n, K^{L_n}_n,\epsilon_n)$-invariant non-negative Borel function $g_n$ on $G$. Then $(g_n)$ is a Borel approximate invariant density and $G$ is Borel amenable.
\end{lemma}

\begin{proof} It is clear that the sequence $(g_n)$ satisfies the conditions of \defnref{Borel approximate density}.
\end{proof}

\begin{corollary}\label{Cauchy}  Let $(G,\lambda)$ be a Borel groupoid endowed with a Borel Haar system. Suppose that there exists an increasing  and exhausting sequence $(E_n)_{n\in\N}$ of Borel subsets of $G$ such that
\begin{enumerate}
\item for all $n\in\N$ and for all $x\in G^{(0)}$, $1\le\lambda^x(E_n)<\infty$;
\item for all $m,n\in\N$, $E_mE_n\subset E_{m+n}$;
\item  $\displaystyle{\lambda^x(E_n)}^{1/n}$ tends to 1 uniformly on  $G^{(0)}$ when $n$ goes to infinity;
\item  $\displaystyle{\lambda^{r(\gamma)}(E_n)\over\lambda^{s(\gamma)}(E_n)}$ tends to 1 uniformly on  $G$  when $n$ goes to infinity.
\end{enumerate}
Then $G$ is Borel amenable.
\end{corollary}

\begin{proof}
Let $k\in\N^*$ and $\epsilon>0$ be given. Choose $\rho>1$ such that $1+\rho-2/\rho^4\le \epsilon/2$. Because of $(iv)$, there exists $N_1\ge k$ such that for $j\ge N_1$ and for all $\gamma\in G$, 
\begin{equation}
\label{eq:3.2}
(1/\rho)\lambda^{s(\gamma)}(E_{kj})\le \lambda^{r(\gamma)}(E_{kj})\le \rho\lambda^{s(\gamma)}(E_{kj})
\end{equation}
Because of $({\rm iii})$, there exists $N\ge 2N_1$ such that for all $x\in G^{(0)}$,
  $$ (\lambda^x(E_{kN}))^{1/N}\le \rho$$
We write $a_n(x)=\lambda^x(E_{kn})$, $A_n(x)=\log a_n(x)$ and $B_n(x)=A_n(x)-A_{n-1}(x)$. Then we have $\sum_{j=1}^N B_j(x)=A_N(x)-A_0(x)\le A_N(x)$. Therefore  $${1\over N-N_1}\sum_{j=N_1+1}^N B_j(x)\le {2\over N}\sum_{j=1}^N B_j(x)\le 2\log\rho$$
This implies that for each $x\in G^{(0)}$, there exists at least one $j\in\{N_1+1,\ldots,N\}$ such that $B_j(x)\le 2\log\rho$. For each $j\in\{N_1+1,\ldots,N\}$, we define $L_j$ as the set of $x$'s such that $B_j(x)\le 2\log\rho$. Then $(L_j)_{j=N_1+1,\ldots,N}$ is a Borel cover of $G^{(0)}$ and
\begin{equation}
\label{eq:3.3}
\forall x\in L_j,\quad \displaystyle{\frac{\lambda^x(E_{kj})}{\lambda^x(E_{k(j-1)})}\le \rho^2}
\end{equation}
Therefore, for $j\in\{N_1+1,\ldots,N\}$ and $\gamma\in E_k^{L_j}$, we obtain by using the inequalities \eqref{eq:3.1}, \eqref{eq:3.2} and \eqref{eq:3.3}:
$$\begin{array}{ccl}\displaystyle{
\frac{\lambda^{r(\gamma)}(\gamma E_{kj}\Delta E_{kj})}{\lambda^{s(\gamma)}(E_{kj})}}&=&1+ \displaystyle{
\frac{\lambda^{r(\gamma)}(E_{kj})}{\lambda^{s(\gamma)}(E_{kj})}}-2\displaystyle{
\frac{\lambda^{r(\gamma)}(\gamma E_{kj}\cap E_{kj})}{\lambda^{s(\gamma)}(E_{kj})}}\\
&\le&1+ \displaystyle{
\frac{\lambda^{r(\gamma)}(E_{kj})}{\lambda^{s(\gamma)}(E_{kj})}}-2 \displaystyle{\frac{\lambda^{s(\gamma)}(E_{k(j-1)})}{\lambda^{s(\gamma)}(E_{kj})}}\\
&\le&1+ \displaystyle{
\frac{\lambda^{r(\gamma)}(E_{kj})}{\lambda^{s(\gamma)}(E_{kj})}}-2\displaystyle{
\frac{\lambda^{s(\gamma)}(E_{k(j-1)})}{\lambda^{r(\gamma)}(E_{k(j-1)})}}\displaystyle{\frac{\lambda^{r(\gamma)}(E_{k(j-1)})}{\lambda^{r(\gamma)}(E_{kj})}}\displaystyle{
\frac{\lambda^{r(\gamma)}(E_{kj})}{\lambda^{s(\gamma)}(E_{kj})}}\\
 &\le&1+\rho-2(1/\rho)(1/\rho^2)(1/\rho)\\
 &\le& \epsilon/2 \end{array}$$
 Applying \lemref{estimate} with $F=E_{kj}$, we obtain a non-negative Borel function $g_j$ which is $(L_j, E_k^{L_j},\epsilon)$-invariant. Applying \lemref{glueing}, we obtain a non-negative Borel function $g$ which is $(G^{(0)}, E_k,\epsilon)$-invariant. Applying \lemref{aim} with $L_n=G^{(0)}$, $K_n=E_n$ and a sequence $(\epsilon_n)$ of real positive numbers decreasing to 0,  we obtain that $G$ is Borel amenable.
\end{proof}

\begin{definition} Let $G$ be a Borel groupoid. A {\it length function} is a Borel map
$l : G \ra \mathbf{R}^+$ such that $l(G^{(0)}) = \{0\}$ and 
\begin{itemize}
\item[{\rm(a)}] $l(\gamma^{-1}) = l(\gamma)$ for all $\gamma \in G$;
\item[{\rm(b)}] $l(\gamma_1 \gamma_2) \leq l(\gamma_1) + l(\gamma_2)$ 
when $s(\gamma_1) = r(\gamma_2)$.
\end{itemize}
If $G$ is endowed with a Haar system $\lambda$, we say that the length function $l$ is {\it proper} if for all $1\le c<\infty$ and all $x\in G^{(0)}$, $0<\lambda^x(B(c))<\infty$, where $B(c)$ is the ball $ \{ \gamma\in G : l(\gamma) \leq c\}$.
\end{definition}

\begin{corollary} Let $(G,\lambda)$ be a Borel groupoid with a Borel Haar system and let $l : G \ra \mathbf{R}^+$ be a proper length function. Let $B(n)$ denote the ball of radius $n$. Assume one of the following conditions
\begin{enumerate}
\item as $n$ goes to $\infty$, $\displaystyle \frac{\lambda^x(B(n+1))}{\lambda^x(B(n))}$ goes to 1 pointwise ;
\item as $n$ goes to $\infty$, ${\lambda^x(B(n))}^{1/n}$ goes to 1 uniformly on $G^{(0)}$ and $\displaystyle{\frac{\lambda^{r(\gamma)}(B(n))}{\lambda^{s(\gamma)}(B(n))}}$ goes to 1 uniformly on $G$.
\end{enumerate}
Then $G$ is Borel amenable.
\end{corollary}

\vskip 5mm
{\it Acknowledgements.} I thank D. Williams for helping me to clarify the notion of topological amenability, for stimulating discussions and for improvement of the manuscript. I am grateful to him and to Dartmouth College for a visit in Summer 2011 which initiated the present work.

\bibliographystyle{amsplain}

\end{document}